\documentclass[12pt,reqno]{amsart}
\usepackage{amscd,amsmath,amsthm,amssymb}
\usepackage{color}
\usepackage{tikz}
\usepackage{url}
\usepackage{circuitikz}
\usepackage{float}

\usepackage{lipsum}

\newtheorem{Theorem}{Theorem}[section]
\newtheorem{Lemma}[Theorem]{Lemma}
\newtheorem{Corollary}[Theorem]{Corollary}
\newtheorem{Proposition}[Theorem]{Proposition}
\newtheorem{Remark}[Theorem]{Remark}

\newtheorem{Conjecture}{Conjecture}
\newtheorem{Thm}[Conjecture]{Theorem}

\def\qed{\ifhmode\textqed\fi
	\ifmmode\ifinner\quad\qedsymbol\else\dispqed\fi\fi}
\def\textqed{\unskip\nobreak\penalty50
	\hskip2em\hbox{}\nobreak\hfill\qedsymbol
	\parfillskip=0pt \finalhyphendemerits=0}
\def\dispqed{\rlap{\qquad\qedsymbol}}

\def\m{\mathfrak{m}}

\def\Ass{\operatorname{Ass}}

\def\depth{\operatorname{depth}}
\def\lcm{\operatorname{lcm}}

\def\supp{\operatorname{supp}}
\def\reg{\operatorname{reg}}
\def\v{\operatorname{v}}
\def\vstab{\operatorname{vstab}}

\def\dstab{\operatorname{dstab}}
\def\astab{\operatorname{astab}}

\def\Tor{\operatorname{Tor}}

\begin{document}
	
	\title{The stable set of associated primes of a complementary edge ideal}
	\author{Antonino Ficarra}
	
	\address{Antonino Ficarra, BCAM -- Basque Center for Applied Mathematics, Mazarredo 14, 48009 Bilbao, Basque Country -- Spain, Ikerbasque, Basque Foundation for Science, Plaza Euskadi 5, 48009 Bilbao, Basque Country -- Spain}
	\email{aficarra@bcamath.org,\,\,\,\,\,\,\,\,\,\,\,\,\,antficarra@unime.it}
	
	
	\begin{abstract}
	We explicitly determine the associated primes of every power of a complementary edge ideal, prove that they satisfy the persistence property, and compute the v-function. In the course of the proofs, we completely describe the homological properties of all powers of squarefree monomial ideals generated in degrees large relative to the number of variables defining them.
	\end{abstract}
	
	\maketitle
	\section*{Introduction}
	
	In 1979, Brodmann \cite{B79} proved that the set of associated primes $\Ass(I^k)$ of the powers of an ideal $I$ in a Noetherian ring eventually becomes stable: there exists an integer $k_0>0$ such that $\Ass(I^{k})=\Ass(I^{k_0})$ for all $k\ge k_0$. The smallest integer $k_0$ for which this property holds is called the \textit{index of ass stability} of $I$ and is denoted by $\astab(I)$. Similarly, the set $\Ass(I^{k_0})$ is called the \textit{stable set of associated primes} of $I$, and it is denoted by $\Ass^\infty(I)$.
	
	Let $S=K[x_1,\dots,x_n]$ be the standard graded polynomial ring over a field $K$, and let $I\subset S$ be a monomial ideal. Then, the associated primes of $I$ are monomial prime ideals. Let $[n]=\{1,\dots,n\}$. A monomial prime ideal is an ideal of the form $P_F=(x_i:\ i\in F)$ for a subset $F\subset[n]$. A very general and effective criterion for determining when $P_F\in\Ass^\infty(I)$ was given in \cite{BHR12,BHR11}, see also \cite[Theorem 2.1]{FSPackA}. Since there are only finitely many subsets of $[n]$, determining $\astab(I)$ and $\Ass^\infty(I)$ for a monomial ideal $I\subset S$ are feasible problems to address. 
	
	In general, if $P_F\in\Ass(I^k)$ it is not necessarily true that $P_F\in\Ass(I^{k+1})$. We say that $I$ satisfies the \textit{persistence property} if
	$$
	\Ass(I)\subset\Ass(I^2)\subset\Ass(I^3)\subset\cdots.
	$$
	In this case $\Ass^\infty(I)=\bigcup_{k\ge1}\Ass(I^k)$. Not all monomial ideals satisfy the persistence property. On the other hand, Mart\'inez-Bernal, Morey, and Villarreal proved that squarefree monomial ideals generated in degree 2 satisfy the persistence property, see \cite[Theorem 2.15]{MMV}. These ideals are known as \textit{edge ideals} \cite{RV}. Given a finite simple graph $G$ on the vertex set $V(G)=[n]$ with edge set $E(G)$, the \textit{edge ideal} of $G$ is the squarefree monomial ideal of $S$ defined as
	$$
	I(G)\ =\ (x_ix_j:\ \{i,j\}\in E(G)).
	$$
	
	Any squarefree monomial ideal generated in degree 2 is an edge ideal. In \cite{LT}, Lam and Viet Trung described the set $\Ass(I(G)^k)$ for any graph $G$ and all $k\ge1$. 
	
	Let $I\subset S$ be a squarefree monomial ideal equigenerated in degree $n-2$. As noted in \cite{FM1,HQS}, $I$ can be seen as the \textit{complementary edge ideal}
	$$
	I_c(G)\ = ((x_1\cdots x_n)/(x_ix_j):\ \{i,j\}\in E(G))
	$$
	of a finite simple graph $G$ on the vertex set $[n]$. There is a rich interplay between the combinatorics of $G$ and the algebraic properties of $I_c(G)$ as highlighted by the results in \cite{FM,FM1,FM2,HQS,LWZ}.
	
	Given a finite simple graph $G$ and $i\in V(G)$, let $\deg_G(i)=|\{j:\ \{i,j\}\in E(G)\}|$. Then $i\in V(G)$ is an isolated vertex if and only if $\deg_G(i)=0$. Denote by $\widetilde{b}(G)$ the number of bipartite connected components of $G$ having more than one vertex. We denote by $G|_F$ the induced subgraph of $G$ on $F\subset[n]$. Inspired by the remarkable behaviour of the set $\Ass(I(G)^k)$ (see \cite{LT,MMV}), in this paper we prove:
	\begin{Thm}\label{ThmA}
		Let $G$ be a finite simple graph on the vertex set $[n]$ with $n\ge3$. Then $I_c(G)$ satisfies the persistence property:
		$$
		\Ass(I_c(G))\subset\Ass(I_c(G)^2)\subset\Ass(I_c(G)^3)\subset\cdots.
		$$
		Moreover,
		$$
		\Ass^\infty(I_c(G))=\{P_{\{i\}}:\ \deg_G(i)=0\}\cup\{P_F:\ |F|>1,\,\widetilde{b}(G|_F)=0\}.
		$$
		and $\astab(I_c(G))\le n-2$.
	\end{Thm}

    To prove this result, we adopt the following strategy. Let $I\subset S$ be a monomial ideal with minimal monomial generating set $\mathcal{G}(I)$. Given $F\subset[n]$ non-empty, we put $S_F=K[x_i:i\in F]$. We define the $K$-linear map $\varphi_F:S\rightarrow S_F$ by setting $\varphi_F(x_i)=x_i$ for $i\in F$, and $\varphi_F(x_j)=1$ otherwise. The monomial ideal $I(P_F)=\varphi_F(I)\subset S_F$ is called the \textit{monomial localization} of $I$ at $P_F$. Concretely, $I(P_F)$ is the monomial ideal of $S_F$ obtained by replacing in $I$ all variables $x_j$ by $1$ for all $j\notin F$.
    
    We have $P_F\in\Ass(I)$ if and only if $\depth S_F/I(P_F)=0$. In view of this fact, Theorem \ref{ThmA} will be proved as a consequence of the following result. It is well-known to the experts (see \cite{HRV}), but for the sake of record, we state it here explicitly.
    \begin{Thm}\label{ThmB}
    	Let $\mathcal{F}$ be a family of monomial ideals of $S$ satisfying the following two properties.
    	\begin{enumerate}
    		\item[\textup{(a)}] $I(P_F)\in\mathcal{F}$, for all $I\in\mathcal{F}$ and all $F\subset[n]$.
    		\item[\textup{(b)}] The function $k\mapsto\depth S/I^k$ is non-increasing, for all $I\in\mathcal{F}$.
    	\end{enumerate}
    	Then, all ideals $I\in\mathcal{F}$ satisfy the persistence property:
    	$$
    	\Ass(I)\subset\Ass(I^2)\subset\Ass(I^3)\subset\cdots.
    	$$
    	Furthermore
    	$$
    	\Ass^\infty(I)=\bigcup_{k\ge1}\Ass(I^k)=\{P_F:\ \depth S_F/I(P_F)^k=0\ \text{for some}\ k\}.
    	$$
    \end{Thm}
    
    Using this result, one can easily prove that polymatroidal ideals satisfy the persistence property, see \cite[Proposition 2.3]{HRV}.
    
    In order to prove Theorem \ref{ThmA}, we employ Theorem \ref{ThmB} together with the following Theorem \ref{ThmC}. This theorem significantly extends several results proved in \cite{FM,FM1,FM2}.
    
    To state this result, we need more notation.\smallskip
    
    Given a finite simple graph $G$, we denote by $c(G)$ the number of connected components of $G$ having more than one vertex.
    
    Let $I\subset S$ be any ideal. Brodmann \cite{B79a} proved that the function $k\mapsto\depth S/I^k$ is eventually constant. The least integer $k_*>0$ such that $\depth S/I^k=\depth S/I^{k_*}$ for all $k\ge k_*$ is called the \textit{index of depth stability} of $I$ and is denoted by $\dstab(I)$.
    
    Given a homogeneous ideal $I\subset S$ and an integer $j\ge0$, we denote by $I_{\langle j\rangle}$ the homogeneous ideal generated by all homogeneous polynomials of degree $j$ belonging to $I$. We say that $I$ is componentwise linear if $I_{\langle j\rangle}$ has linear resolution for all $j\ge0$.
    
    Let $I\subset S$ be a monomial ideal. We say that $I$ has \textit{linear quotients} if there exists an order $u_1,\dots,u_m$ of $\mathcal{G}(I)$ such that $(u_1,\dots,u_{j-1}):u_j$ is generated by variables, for all $j=1,\dots,m$. A monomial ideal with linear quotients is componentwise linear \cite[Theorem 8.2.15]{HHBook} but the converse is not necessarily true.\smallskip
    
    Now, we are in a position to state Theorem \ref{ThmC}.
    
    \begin{Thm}\label{ThmC}
    	Let $I\subset S$ be a squarefree monomial ideal generated in degrees $\ge n-2$. Then, the following statements hold.
    	\begin{enumerate}
    		\item[\textup{(a)}] The function $k\mapsto\depth S/I^k$ is non-increasing.
    		\item[\textup{(b)}] $\operatorname{dstab}(I)\le n-1$.
    		\item[\textup{(c)}] \begin{enumerate}
    			\item[\textup{(i)}] If $I$ is equigenerated in degree $d\in\{n-1,n\}$, then
    			$$
    			\reg I^k\ =\ dk,\quad\textit{for all}\ k\ge1.
    			$$
    			\item[\textup{(ii)}] If $I$ is equigenerated in degree $n-2$, let $G$ be the unique simple graph on $[n]$ such that $I=I_c(G)$, and put $c=c(G)$. Then
    			$$
    			\reg I^k\ =\ \begin{cases}
    				(n-1)k&\textit{for}\ 1\le k\le c-2,\\
    				(n-2)k+c-1&\textit{for}\ k\ge c-1.
    			\end{cases}
    			$$
    			\item[\textup{(iii)}] If $I$ is minimally generated in degrees $n-2$ and $n-1$, then
    			$$
    			\reg I^k\ =\ (n-1)k,\quad\textup{for all}\ k\ge1.
    			$$
    		\end{enumerate}
    		\item[\textup{(d)}] The following conditions are equivalent.
    		\begin{enumerate}
    			\item[\textup{(i)}] $I^k$ has linear quotients for all $k\ge1$.
    			\item[\textup{(ii)}] $I^k$ is componentwise linear for all $k\ge1$.
    			\item[\textup{(iii)}] Either $I_{\langle n-2\rangle}=0$ or $I_{\langle n-2\rangle}=I_c(G)\ne0$ and $c(G)=1$.
    		\end{enumerate}
    		\item[\textup{(e)}] The graded Betti numbers of $I^k$ do not depend on $K$, for all $k\ge1$.
    	\end{enumerate}
    \end{Thm}
    
    The paper is organized as follows.
    
    In Section \ref{sec1} we give the proof of Theorem \ref{ThmB} and as a consequence we show that all squarefree monomial ideals of $S$ generated in degrees $\ge n-2$ satisfy the persistence property. Section \ref{sec2} is devoted to the proof of Theorem \ref{ThmC}. Next, in Section \ref{sec3}, we prove Theorem \ref{ThmA}.
    
    Section \ref{sec4} contains some applications of Theorem \ref{ThmA}. In Corollary \ref{Cor:RoySaha}, we give a simple proof of the main result of \cite{RS}. Namely, we classify the complementary edge ideals whose ordinary and symbolic powers coincide. In Theorem \ref{Thm:v-numb-I_c(G)}, we compute the $\v$-function of any complementary edge ideal.
    
	\section{Criteria for the persistence property}\label{sec1}
	
	In this short section, we prove Theorem \ref{ThmB}.
    \begin{proof}[Proof of Theorem \ref{ThmB}]
    	Let $I\in\mathcal{F}$ and $P_F\in\Ass_S(I^k)$ for some $F\subset[n]$ and $k\ge1$. Then $P_F$ is the maximal ideal of $S_F$ and $P_F\in\Ass_{S_F}(I^k(P_F))=\Ass_{S_F}(I(P_F)^k)$. This is equivalent to the fact that $\depth S_F/I(P_F)^k=0$. By property (a), $I(P_F)\in\mathcal{F}$. Note that $\depth S/I(P_F)^k=\depth S_F/I(P_F)^k+n-|F|$. By property (b), the depth function $k\mapsto\depth S/I(P_F)^k$ is non-increasing. It follows that $\depth S_F/I(P_F)^k$ is non-increasing too. Hence, $\depth S_F/I(P_F)^\ell=0$, and so $P_F\in\Ass_{S_F}(I(P_F)^\ell)=\Ass_{S_F}(I^\ell(P_F))$, for all $\ell\ge k$. Therefore $P_F\in\Ass_S(I^\ell)$ for all $\ell\ge k$, and so $I$ satisfies the persistence property. This argument also proves the second assertion.
    \end{proof}
    
    Theorem \ref{ThmB} can be used to show that some classes of monomial ideals satisfy the persistence property. For instance, polymatroidal ideals, see \cite[Proposition 2.3]{HRV}.
    
    As the main new application of Theorem \ref{ThmB}, we show
    \begin{Theorem}
    	Let $I\subset S$ be a squarefree monomial ideal generated in degrees $\ge n-2$. Then $I$ satisfies the persistence property.
    \end{Theorem}
    \begin{proof}
    	Recall that the \textit{support} of a monomial ideal $I\subset S$ is defined as the set $\supp(I)=\bigcup_{u\in\mathcal{G}(I)}\supp(u)$, where the support of a monomial $u\in S$ is defined as the set $\supp(u)=\{i:\ x_i\ \textup{divides}\ u\}$.
    	
    	Let us define the family
    	$$
    	\mathcal{F}\ =\ \bigg\{I\subset S\,\,\bigg|\,\, \substack{\displaystyle I\ \textup{is a squarefree monomial ideal}\\[3pt] \displaystyle\textup{generated in degrees}\ \ge\ |\supp(I)|-2}\,\bigg\}.
    	$$
    	
    	In view of Theorem \ref{ThmB}, our assertion follows once we show that the properties (a) and (b) in Theorem \ref{ThmB} are satisfied by the family $\mathcal{F}$.
    	
    	For a non-empty subset $A\subset[n]$, we put ${\bf x}_A=\prod_{i\in A}x_i$. Let $I\in\mathcal{F}$ and put $G=\supp(I)\subset[n]$. Let $i,j\in G$ be distinct elements. Let $F\subset[n]$. We note that
    	$$
    	\varphi_F({\bf x}_G)\ =\ {\bf x}_{G\cap F},\quad\quad\quad \varphi_F({\bf x}_{G\setminus\{i\}})\ =\ \begin{cases}
    		\,{\bf x}_{G\cap F}&\textup{if}\ i\notin G,\\
    		\,{\bf x}_{(G\cap F)\setminus\{i\}}&\textup{if}\ i\in G,
    	\end{cases}
    	$$
    	$$
    		\varphi_F({\bf x}_{G\setminus\{i,j\}})\ =\ \begin{cases}
    			\,{\bf x}_{G\cap F}&\textup{if}\ i,j\notin G,\\
    			\,{\bf x}_{(G\cap F)\setminus\{i\}}&\textup{if}\ i\in G,\,j\notin G,\\
    			\,{\bf x}_{(G\cap F)\setminus\{j\}}&\textup{if}\ i\notin G,\,j\in G,\\
    			\,{\bf x}_{(G\cap F)\setminus\{i,j\}}&\textup{if}\ i,j\in G.
    		\end{cases}
    	$$\smallskip
    	
    	\noindent Since $\supp(I(P_F))\subset\supp(I)\cap F=G\cap F$, the above computations show that $I(P_F)$ is a squarefree monomial ideal generated in degrees $\ge|G\cap F|-2\ge|\supp(I(P_F))|-2$. Hence $I(P_F)\in\mathcal{F}$ and the property (a) is satisfied.
    	
    	The property (b) follows from statement (a) of Theorem \ref{ThmC}.
    \end{proof}
	
	\section{Squarefree monomial ideals generated in big degrees}\label{sec2}
	
	Let $S=K[x_1,\dots,x_n]$ with $n\ge3$. Our goal is to prove Theorem \ref{ThmC}. To this end, note that for a squarefree monomial ideal $I\subset S$ generated in degrees $\ge n-2$ only one of the following possibilities occurs:
	\begin{enumerate}
		\item[(i)] $I$ is generated in degree $n-2$, in which case it is a complementary edge ideal,
		\item[(ii)] $I$ is generated in degree $d\in\{n-1,n\}$ in which case it is a matroidal ideal,
		\item[(iii)] $I$ is minimally generated in two degrees $n-2$ and $n-1$.
	\end{enumerate}

	In the case (i), all the claims of Theorem \ref{ThmC} follow from the results in \cite{FM,FM1,FM2}. The case (ii) is trivial as well, because then $I$ is a matroidal ideal, see \cite{HQ,HRV}. Only the case (iii) remains to be addressed. This is accomplished next.
	
	Given a monomial ideal $I\subset S$ and $j\ge0$, let $\mu_j(I)=|\{u\in\mathcal{G}(I):\deg(u)=j\}|$.
	
	\begin{Theorem}\label{Thm:reg-deg>=n-2}
		Let $I\subset S$ be a squarefree monomial ideal minimally generated in two degrees $n-2$ and $n-1$. Let $J=I_{\langle n-2\rangle}$. The following statements hold.
		\begin{enumerate}
			\item[\textup{(a)}] $\reg I^k=(n-1)k$ for all $k\ge1$.\smallskip
			\item[\textup{(b)}] The function $k\mapsto\depth S/I^k$ is non-increasing.\smallskip
			\item[\textup{(c)}] The graded Betti numbers of $I^k$ do not depend on $K$, for all $k\ge1$.\smallskip
			\item[\textup{(d)}] $\depth S/I^k=\depth S/J^k-\mu_{n-1}(I)$ for all $k\ge\dstab(J)+\mu_{n-1}(I)$.\smallskip
			\item[\textup{(e)}] $\dstab(I)\le n-2$.\smallskip
			\item[\textup{(f)}] The following conditions are equivalent.\smallskip
			\begin{enumerate}
				\item[\textup{(i)}] $I^k$ has linear quotients for all $k\ge1$.
				\item[\textup{(ii)}] $I^k$ is componentwise linear for all $k\ge1$.
				\item[\textup{(iii)}] $J=I_c(G)$ with $c(G)=1$.
			\end{enumerate}
		\end{enumerate}
	\end{Theorem}
	
	Let $I,I_1,I_2\subset S$ be monomial ideals such that $\mathcal{G}(I)=\mathcal{G}(I_1)\sqcup\mathcal{G}(I_2)$. Following \cite[Definition 1.1]{FHT}, we say that $I=I_1+I_2$ is a \textit{Betti splitting} if
	\begin{equation}\label{eq:BettiSplitEq}
		\beta_{i,j}(I)\ =\ \beta_{i,j}(I_1)+\beta_{i,j}(I_2)+\beta_{i-1,j}(I_1\cap I_2),\ \ \ \text{for all}\ i,j\ge0.
	\end{equation}
	In this case, by \cite[Corollary 2.2]{FHT} we have
	\begin{align*}
		\depth S/I\ &=\ \min\{\depth S/I_1,\,\depth S/I_2,\,\depth S/(I_1\cap I_2)-1\},\\
		\reg I\ &=\ \max\{\reg I_1,\,\reg I_2,\,\reg I_1\cap I_2-1\}.
	\end{align*}
	
	Consider the natural short exact sequence $0\rightarrow I_1\cap I_2\rightarrow I_1\oplus I_2\rightarrow I\rightarrow0$. By \cite[Proposition 2.1]{FHT}, equation (\ref{eq:BettiSplitEq}) holds, that is $I=I_1+I_2$ is a Betti splitting, if and only if, the induced maps  
	$$
	\Tor_i^S(K,I_1\cap I_2)\rightarrow\Tor_i^S(K,I_1)\oplus\Tor_i^S(K,I_2)
	$$
	are zero for all $i\ge0$. In this case, we say that the maps $I_1\cap I_2\rightarrow I_1$ and $I_1\cap I_2\rightarrow I_2$ are \textit{$\Tor$-vanishing}. We quote \cite[Proposition 3.1]{EK} (or also \cite[Lemma 4.2]{NV}).
	
	\begin{Lemma}\label{lemma:criteria-bs}
		Let $J,L\subset S$ be non-zero monomial ideals with $J\subset L$. Suppose there exists a map $\varphi:\mathcal{G}(J)\rightarrow\mathcal{G}(L)$ such that for any $\emptyset\ne\Omega\subseteq \mathcal{G}(J)$ we have
		$$
		\lcm(u:u\in\Omega)\in \mathfrak{m}\cdot(\lcm(\varphi(u):u\in\Omega)),
		$$
		where $\mathfrak{m}=(x_1,\dots,x_n)$. Then the inclusion map $J\rightarrow L$ is $\Tor$-vanishing.
	\end{Lemma}
	
	Given a monomial ideal $I\subset S$, let $\partial I$ be the ideal generated by the elements $u/x_i$, with $u\in\mathcal{G}(I)$ and $i\in\supp(u)$. The next result follows from \cite[Proposition 4.4]{NV} due to Nguyen and Vu. See also \cite[Lemma 1.3 and Theorem 1.2]{CFL}.
	
	\begin{Lemma}\label{lemma:bs}
		Let $J,L\subset S$ be non-zero monomial ideals. Suppose that $\partial J\subset L$. Then $J\subseteq\mathfrak{m}L$ and there exists a map $\varphi:\mathcal{G}(J)\rightarrow\mathcal{G}(L)$ satisfying the assumptions of Lemma \ref{lemma:criteria-bs}. In particular, the inclusion map $J\rightarrow L$ is $\Tor$-vanishing.
	\end{Lemma}
	
	The next two basic facts will be used without reference in what follows.
	\begin{Remark}
		Let $I\subset S$ be a graded ideal and let $f\in S$ be a homogeneous element. Then $\reg fI=\reg I+\deg(f)$ and $\depth S/(fI)=\depth S/I$.
	\end{Remark}
	\begin{Remark}
		Let $I_x \subset S_x=K[x_1,\ldots,x_p]$ and $I_y\subset S_y=K[y_1,\ldots,y_q]$ be proper monomial ideals in polynomial rings over disjoint sets of variables. Then, we have $\reg I_xI_y=\reg I_x+\reg I_y$ and $\depth S/(I_xI_y)=\depth S_x/I_x+\depth S_y/I_y+1$. Moreover, if $I_y=(1)$ but $I_x\ne(1)$, then $\depth S/(I_xI_y)=\depth S_x/I_x+q$. Similarly, if $I_x=(1)$ but $I_y\ne(1)$, then $\depth S/(I_xI_y)=\depth S_y/I_y+p$.
	\end{Remark}

    For a monomial $u\in S$ and an integer $i\in[n]$, the $x_i$-degree of $u$ is defined as
    $$
    \deg_{x_i}(u)\ =\ \max\{j:\ x_i^j\ \textup{divides}\ u\}.
    $$
	
	Now, we are in a position to prove Theorem \ref{Thm:reg-deg>=n-2}.
	\begin{proof}[Proof of Theorem \ref{Thm:reg-deg>=n-2}]
		Note that $I_{\langle n-2\rangle}=I_c(G)$ is the complementary edge ideal of some graph $G$. Since $I$ is minimally generated in the degrees $n-2$ and $n-1$, we can find a minimal monomial generator $u\in\mathcal{G}(I)$ of degree $n-1$. Then $u={\bf x}_{[n]}/x_i$ for some $i\in[n]$. We claim that $i$ is an isolated vertex of $G$. Otherwise, we would have $\{i,j\}\in E(G)$ for some $j$, and then ${\bf x}_{[n]}/(x_ix_j)\in I_{\langle n-2\rangle}$ would divide $u$ against the fact that $u$ is a minimal generator. Let $F=\{i:\ {\bf x}_{[n]}/x_i\in\mathcal{G}(I)\}$. Then $F$ is a subset of the isolated vertices of $G$, and therefore ${\bf x}_F$ divides all the monomial generators of $I_c(G)$. Up to relabeling, we may assume that $F=[p]$ for some $p\ge1$. Let $H=G|_{[n]\setminus[t]}$. Then $I_c(G)={\bf x}_{[p]}I_c(H)$. For our convenience, we relabel the variables $\{x_{p+1},\dots,x_n\}$ by $\{y_1,\dots,y_q\}$. Our discussion then shows that we can write
		$$
		I\ =\ {\bf x}_{[t]}I_c(H)+{\bf y}_{[q]}({\bf x}_{[t]}/x_i:\ i\in[t]).
		$$
		We put $I_y=I_c(H)$ and $I_x=({\bf x}_{[t]}/x_i:\ i\in[t])$. Then
		$$
		I^k\ = \ \sum_{h=0}^k{\bf x}_{[p]}^{k-h}{\bf y}_{[q]}^hI_y^{k-h}I_x^h
		$$
		for all $k\ge1$. Fix an integer $k\ge1$, and set $J_\ell=\sum_{h=0}^\ell{\bf x}_{[p]}^{k-h}{\bf y}_{[q]}^hI_y^{k-h}I_x^h$ for $0\le\ell\le k$.
		
		Notice that $J_k=I^k$. For each $0<\ell\le k$, we claim that
		\begin{equation}\label{eq:Iell-bs}
			J_\ell\ =\ J_{\ell-1}+{\bf x}_{[p]}^{k-\ell}{\bf y}_{[q]}^{\ell}I_y^{k-\ell}I_x^{\ell}
		\end{equation}
		is a Betti splitting of $J_\ell$. First, we note that
		$$
		\mathcal{G}(J_\ell)\ =\ \mathcal{G}(J_{\ell-1})\sqcup\mathcal{G}({\bf x}_{[p]}^{k-\ell}{\bf y}_{[q]}^{\ell}I_y^{k-\ell}I_x^{\ell}).
		$$
		
		To this end, let $u\in\mathcal{G}(J_{\ell-1})$ and $v\in\mathcal{G}({\bf x}_{[p]}^{k-\ell}{\bf y}_{[q]}^{\ell}I_y^{k-\ell}I_x^{\ell})$. Then $u\in{\bf x}_{[p]}^{k-h}{\bf y}_{[q]}^{h}I_y^{k-h}I_x^{h}$ for some $0\le h<\ell$. A quick computation shows that
		\begin{align*}
			\sum_{i\in[p]}\deg_{x_i}(u)\ =\ pk-h\ &>\ pk-\ell\ =\ \sum_{i\in[p]}\deg_{x_i}(v),\\
			\sum_{j\in[q]}\deg_{y_j}(u)\ =\ (q-2)k+2h\ &<\ (q-2)k+2\ell\ =\ \sum_{j\in[q]}\deg_{y_j}(v).
		\end{align*}
	    Hence $u$ can not divide $v$ and $v$ can not divide $u$. This proves our claim. Next, we compute the intersection
		\begin{align*}
			L\ =\ J_{\ell-1}\cap({\bf x}_{[p]}^{k-\ell}{\bf y}_{[q]}^{\ell}I_y^{k-\ell}I_x^{\ell})\ &=\ (\sum_{h=0}^{\ell-1}{\bf x}_{[p]}^{k-h}{\bf y}_{[q]}^{h}I_y^{k-h}I_x^{h})\cap({\bf x}_{[p]}^{k-\ell}{\bf y}_{[q]}^{\ell}I_y^{k-\ell}I_x^{\ell})\\
			&=\phantom{\ (}\sum_{h=0}^{\ell-1}[({\bf x}_{[p]}^{k-h}{\bf y}_{[q]}^{h}I_y^{k-h}I_x^{h})\cap({\bf x}_{[p]}^{k-\ell}{\bf y}_{[q]}^{\ell}I_y^{k-\ell}I_x^{\ell})]\\
			&=\phantom{\ (}\sum_{h=0}^{\ell-1}[({\bf x}_{[p]}^{k-h}I_x^{h})\cap({\bf x}_{[p]}^{k-\ell}I_x^{\ell})][({\bf y}_{[q]}^{h}I_y^{k-h})\cap({\bf y}_{[q]}^{\ell}I_y^{k-\ell})].
		\end{align*}
		
		Note that $({\bf x}_{[p]})\subset I_x$ and $({\bf y}_{[q]})\subset I_y$. Given $h<\ell$, we have $k-h>k-\ell$ and so ${\bf x}_{[p]}^{k-h}I_x^h={\bf x}_{[p]}^{k-\ell}{\bf x}_{[p]}^{\ell-h}I_x^h\subset{\bf x}_{[p]}^{k-\ell}I_x^{\ell-h}I_x^h={\bf x}_{[p]}^{k-\ell}I_x^\ell$. Similarly, we have ${\bf y}_{[q]}^\ell I_y^{k-\ell}\subset{\bf y}_{[q]}^h I_y^{k-h}$. Consequently,
		\begin{align*}
			L\ =\ J_{\ell-1}\cap({\bf x}_{[p]}^{k-\ell}{\bf y}_{[q]}^{\ell}I_y^{k-\ell}I_x^{\ell})\ &=\  \sum_{h=0}^{\ell-1}({\bf x}_{[p]}^{k-h}I_x^h{\bf y}_{[q]}^\ell I_y^{k-\ell})\ =\, (\sum_{h=0}^{\ell-1}{\bf x}_{[p]}^{k-h}I_x^h)({\bf y}_{[q]}^\ell I_y^{k-\ell})\\[3pt]
			&=\ {\bf x}_{[p]}^{k-\ell+1}I_x^{\ell-1}{\bf y}_{[q]}^\ell I_y^{k-\ell},
		\end{align*}
		where we used that $({\bf x}_{[p]}^k)\subset{\bf x}_{[p]}^{k-1}I_x\subset\cdots\subset{\bf x}_{[p]}^{k-(\ell-1)}I_x^{\ell-1}$.
		
		To conclude that (\ref{eq:Iell-bs}) is indeed a Betti splitting we must prove that the inclusion maps $L\rightarrow J_{\ell-1}$ and $L\rightarrow{\bf x}_{[p]}^{k-\ell}{\bf y}_{[q]}^{\ell}I_y^{k-\ell}I_x^{\ell}$ are $\Tor$-vanishing.\smallskip
		
		We show that the first map is $\Tor$-vanishing, the other case is similar. First, suppose for a moment that $H$ does not contain isolated vertices. Under this assumption, we claim that $\partial({\bf y}_{[q]}^\ell I_y^{k-\ell})\subset{\bf y}_{[q]}^{\ell-1} I_y^{k+1-\ell}$. Indeed, take $u={\bf y}_{[q]}^\ell u_1\cdots u_{k-\ell}\in\mathcal{G}({\bf y}_{[q]}^\ell I_y^{k-\ell})$ with $u_s\in\mathcal{G}(I_y)$ for all $s=1,\dots,k-\ell$, and let $y_i$ divide $u$. Since we assumed that $H$ does not contain isolated vertices, we have ${\bf y}_{[q]}/(y_iy_j)\in\mathcal{G}(I_y)$ for some $j$. Then
		$$
		{\bf y}_{[q]}^{\ell-1}u_1\cdots u_{k-\ell}({\bf y}_{[q]}/(y_iy_j))\ \ \textup{divides}\ \ u/y_i.
		$$
		
		This shows that $u/y_i\in {\bf y}_{[q]}^{\ell-1} I_y^{k+1-\ell}$. Hence $\partial({\bf y}_{[q]}^\ell I_y^{k-\ell})\subset{\bf y}_{[q]}^{\ell-1} I_y^{k+1-\ell}$. Lemma \ref{lemma:bs} implies that there exists a map $\varphi:\mathcal{G}({\bf y}_{[q]}^\ell I_y^{k-\ell})\rightarrow\mathcal{G}({\bf y}_{[q]}^{\ell-1} I_y^{k+1-\ell})$ such that for any non-empty subset $\Omega\subset\mathcal{G}({\bf y}_{[q]}^\ell I_y^{k-\ell})$ we have
		\begin{equation}\label{eq:lcm}
			\lcm(u:u\in\Omega)\ \in\ \m_y\cdot(\lcm(\varphi(u):u\in\Omega)),
		\end{equation}
		where $\m_y=(y_1,\dots,y_q)$. We claim that such a map $\Phi$ exists also when $H$ contains isolated vertices.
		
		To this end, assume that $H$ contains isolated vertices. Say that these vertices correspond to the variables $y_1,\dots,y_r$ for some $r<q$. Then ${\bf y}_{[q]}^\ell I_y^{k-\ell}={\bf y}_{[r]}^{\ell}{\bf y}_{[q]\setminus[r]}^{\ell}(I_y')^{k-\ell}$ where $I_y'=I_c(H')$ and $H'$ is the graph obtained from $H$ by removing its isolated vertices. Similarly, ${\bf y}_{[q]}^{\ell-1} I_y^{k-\ell+1}={\bf y}_{[r]}^{\ell-1}{\bf y}_{[q]\setminus[r]}^{\ell-1}(I_y')^{k-\ell+1}$. Repeating the argument as before, we obtain a map $\Phi':\mathcal{G}({\bf y}_{[q]\setminus[r]}^{\ell}(I_y')^{k-\ell})\rightarrow\mathcal{G}({\bf y}_{[q]\setminus[r]}^{\ell-1}(I_y')^{k-\ell+1})$ satisfying the property (\ref{eq:lcm}). Then, we define the map $\varphi:\mathcal{G}({\bf y}_{[q]}^\ell I_y^{k-\ell})\rightarrow\mathcal{G}({\bf y}_{[q]}^{\ell-1} I_y^{k+1-\ell})$ by setting $\Phi(u)={\bf y}_{[r]}^{\ell-1}\Phi(u/{\bf y}_{[r]}^{\ell})$. It follows immediately that the property (\ref{eq:lcm}) is again satisfied.\smallskip
		
		Having acquired the existence of the map $\Phi$ for any graph $H$, we define a map
		$$
		\Phi:\mathcal{G}(L)\rightarrow\mathcal{G}(J_{\ell-1})
		$$
		using the map $\varphi$. Let $u\in\mathcal{G}(L)$. Then $u=u_xu_y$, where $u_x\in\mathcal{G}({\bf x}_{[p]}^{k+1-\ell}I_x^{\ell-1})$ and $u_y\in\mathcal{G}({\bf y}_{[q]}^\ell I_y^{k-\ell})$. We put $\Phi(u)=\varphi(u_y)u_x$. Note that $\Phi$ is well-defined, because $\Phi(u)\in\mathcal{G}(J_{\ell-1})$. Finally, let $\Omega\subset\mathcal{G}(L)$ be a non-empty subset. We can write $\lcm(u:u\in\Omega)= \lcm(u_x:u\in\Omega)\cdot\lcm(u_y:u\in\Omega)$ and by using (\ref{eq:lcm}) we have
		\begin{align*}
			\lcm(u:u\in\Omega)\ \in&\ \m_y\cdot(\lcm(\varphi(u_y):u\in\Omega))\cdot\lcm(u_x:u\in\Omega)\\
			=&\ \m_y\cdot(\lcm(\varphi(u_y)u_x:u\in\Omega))\\
			\subset&\ \m\cdot(\lcm(\Phi(u):u\in\Omega)),
		\end{align*}
		where $\m=(x_1,\dots,x_p,y_1,\dots,y_q)$. Lemma \ref{lemma:criteria-bs} implies that the inclusion $L\rightarrow J_{\ell-1}$ is indeed $\Tor$-vanishing, as desired.\smallskip
		
		Let $S_x=K[x_1,\dots,x_p]$ and $S_y=K[y_1,\dots,y_q]$. Next, we claim that
		\begin{align}
			\label{eq:regbs}\reg I^k=\max&\left\{\substack{\displaystyle\reg I_y^k+pk,(n-1)k,\\[5pt]\displaystyle\max_{0<h<k}\{\reg I_y^{k-h}\!+pk+(q-1)h\}}\right\},\\[4pt]
			\label{eq:depthbs}\depth\frac{S}{I^k}=\min&\left\{\substack{\displaystyle\!\depth S_y/I_y^k\!+p,\depth S_x/I_x^k\!+q,\depth S_x/I_x^{k-1}\!+q-1,\\[5pt]\displaystyle\min_{0<h<k}\{\depth S_y/I_y^{k-h}+\depth S_x/I_x^h+1\},\\[3pt]\displaystyle\min_{0<h<k}\{\depth S_y/I_y^{k-h}+\depth S_x/I_x^{h-1}\}}\right\}.
		\end{align}
		
		Since (\ref{eq:Iell-bs}) is a Betti splitting and $I^k=J_k$, we have
		\begin{align*}
			\reg I^k&=\max\{\reg J_{k-1},\,\reg({\bf y}_{[q]}^kI_x^k),\,\reg({\bf x}_{[p]}{\bf y}_{[q]}^kI_x^{k-1})-1\},\\
			\depth S/I^k&=\min\{\depth S/J_{k-1},\,\depth S/({\bf y}_{[q]}^kI_x^k),\,\depth S/({\bf x}_{[p]}{\bf y}_{[q]}^kI_x^{k-1})-1\}.
		\end{align*}
		
		Since $I_x\subset S_x$ is the squarefree Veronese ideal of $S_x$ generated in degree $p-1$ and $\dim(S_x)=p$, we have $\reg I_x^h=(p-1)h$ and the function $k\mapsto\depth S_x/I_x^h$ is non-increasing, for all $h\ge1$. Hence $\reg({\bf y}_{[q]}^kI_x^k)=qk+(p-1)k=(n-1)k$ for all $k\ge1$. Analogously, $\reg({\bf y}_{[q]}{\bf x}_{[p]}^kI_x^{k-1})-1=(n-1)k$. Similar computations can be performed for the depth. Hence,
		\begin{align*}
			\reg I^k&=\max\{\reg J_{k-1},\,(n-1)k\},\\
			\depth S/I^k&=\min\{\depth S/J_{k-1},\,\depth S_x/I_x^k+q,\,\depth S_x/I_x^{k-1}+q-1\}.
		\end{align*}
		
		Iterating these computations to $J_{k-1},\dots,J_0$ by using the Betti splittings (\ref{eq:Iell-bs}), we see that the formulas (\ref{eq:regbs}) and (\ref{eq:depthbs}) indeed hold.\smallskip
		
		Now, we proceed to prove the statements (a), (b), (c), (d), (e) and (f).\smallskip
		
		(a) Let $c=c(H)$ be the number of connected components of $H$ which are not isolated vertices. By \cite[Theorem 4.1]{FM1} we have
		\begin{equation}\label{IndI1}
			\reg I_y^k\ =\ \begin{cases}
				(q-1)k&\textit{if}\ \ 1\le k\le c-2,\\
				(q-2)k+c-1&\textit{if}\ \ k\ge c-1.
			\end{cases} 
		\end{equation}
		
		From this formula it follows that $\reg I_y^k\le(q-1)k$ for all $k\ge1$. Hence, for all $k\ge1$ and all $0<h<k$, we have
		\begin{align*}
			\reg I_y^k+pk\ &\le\ (q-1)k+pk\ =\ (n-1)k,\\
			\reg I_y^{k-h}+pk+(q-1)h\ &\le\ (q-1)(k-h)+pk+(q-1)h\ =\ (n-1)k.
		\end{align*}
		
		These inequalities and equation (\ref{eq:regbs}) imply that $\reg I^k=(n-1)k$ for all $k\ge1$.\smallskip
		
		(b) By \cite[Theorem 4.1]{FM1}, the function $k\mapsto\depth S_y/I_y^k$ is non-increasing. Moreover $k\mapsto\depth S_x/I_x^k$ is non-increasing as well, because $I_x$ is matroidal. It follows at once from   formula (\ref{eq:depthbs}) that $k\mapsto\depth S/I^k$ is non-increasing too.\smallskip
		
		(c) By \cite[Corollary 4.7]{FM1}, the graded Betti numbers of $I_y^k$ do not depend on the ground field $K$ for all $k\ge1$. The same independence occurs for the powers of $I_x$ because it is a matroidal ideal. Consider the Betti splitting (\ref{eq:Iell-bs}). We computed before that $L=J_{\ell-1}\cap({\bf x}_{[p]}^{k-\ell}{\bf y}_{[q]}^{\ell}I_y^{k-\ell}I_x^{\ell})={\bf x}_{[p]}^{k-\ell+1}I_x^{\ell-1}{\bf y}_{[q]}^\ell I_y^{k-\ell}$. This ideal is the product of the ideals ${\bf x}_{[p]}^{k-\ell+1}I_x^{\ell-1}$ and ${\bf y}_{[q]}^\ell I_y^{k-\ell}$ which have disjoint supports and whose graded Betti numbers do not depend on $K$. Hence the graded Betti numbers of $L$ do not depend on $K$. Finally, the Betti splitting (\ref{eq:Iell-bs}) for $\ell=k$ shows that the graded Betti numbers of $I^k$ do not depend on $K$ for all $k\ge1$.\smallskip
		
		(d) We denote by $b(H)$ the number of bipartite connected components of the graph $H$, we remark that an isolated vertex is regarded as a bipartite connected component. By \cite[Theorem 4.1]{FM2}, $\depth S_y/I_y^k=\depth S_y/I_c(H)^k=b(H)$ for all $k\ge\dstab(I_y)$. Next, note that $\depth S_x/I_x^{p-1}=0$. To this end, consider the monomial $u={\bf x}_{[p]}^{p-2}$. It does not belong to $I_x^{p-1}$, because $\deg(u)=p(p-2)$ is strictly less than $(p-1)^2$ which is the initial degree of $I_x^{p-1}$. On the other hand, for all $i\in[p]$ we have
		$$
		x_iu\ =\ x_i{\bf x}_{[p]}^{p-2}\ =\ \prod_{\substack{j\in[p]\\ j\ne i}}({\bf x}_{[p]}/x_j)\ \in\ I_x^{p-1}.
		$$ 
		Hence $(I_x^{p-1}:u)=(x_1,\dots,x_p)\in\Ass_{S_x}(I_x^{p-1})$ and so $\depth S_x/I_x^{p-1}=0$. Since $I_x$ is a matroidal ideal, we have $\depth S_x/I_x^{k}=0$ as well for all $k\ge p-1=\mu_{n-1}(I)-1$.
		
		Let $k\ge\dstab(I_y)+\mu_{n-1}(I)$. The previous discussion and equation (\ref{eq:depthbs}) yield that
		$$
			\depth S/I^k\ =\ \min\left\{\substack{\displaystyle b(H)+p,\,q,\,q-1,\\[5pt]\displaystyle b(H)+1,\,b(H)}\right\}\ =\ \min\{b(H),\,q-1\}\ =\ b(H),
		$$
	    because $q-1=|V(H)|-1\ge b(H)$ since $H$ has at least one edge.
	    
	    By \cite[Theorem 4.1]{FM2}, $\depth S/I_{\langle n-2\rangle}^k=\depth S/({\bf x}_{[p]}I_y)^k=\depth S/I_c(G)^k=b(G)$ for all $k\ge\dstab(I_y)+\mu_{n-1}(I)$. Recall that $H$ is the graph obtained from $G$ by removing the isolated vertices $1,\dots,p$. Hence $b(H)=b(G)-p$. Altogether, $$\depth S/I^k\ =\ b(H)\ =\ b(G)-p\ =\ \depth S/I_{\langle n-2\rangle}^k-\mu_{n-1}(I)$$ for all $k\ge\dstab(I_{\langle n-2\rangle})+\mu_{n-1}(I)$.\smallskip
	    
	    (e) Part (d) implies that $\dstab(I)\le\dstab(I_y)+p$. By \cite[Theorem 4.1]{FM2}, we have $\dstab(I_y)\le q-2$. Since $p+q-2=n-2$, the assertion follows.\smallskip
	    
	    (f) The implication (i) $\Rightarrow$ (ii) is true for any monomial ideal. Assume (ii). Since $n-2$ is the initial degree of $I$, $J=I_c(G)$ has linear powers. By \cite[Theorem B]{FM} we have $c(G)=1$ and (iii) follows. Finally, we show that (iii) $\Rightarrow$ (i). Assume that $J=I_c(G)$ with $c(G)=1$. Again by \cite[Theorem B]{FM}, $J^k$ has linear quotients for all $k\ge1$. Since (\ref{eq:Iell-bs}) is a Betti splitting, by what we proved before we have
	    \begin{equation}\label{eq:G-disjoint}
	    	\mathcal{G}(I^k)\ =\ \bigsqcup_{h=0}^k\mathcal{G}({\bf x}_{[p]}^{k-h}{\bf y}_{[q]}^hI_y^{k-h}I_x^h),\quad\textup{for all}\ k\ge1.
	    \end{equation}
	    Notice that $c(H)=1$ since $c(G)=1$. Hence $I_y^k=I_c(H)^k$ has linear quotients for all $k\ge1$. Since $I_x$ is a matroidal ideal, then also $I_x^k$ has linear quotients for all $k\ge1$. Since $\supp(I_x)\cap\supp(I_y)=\emptyset$, \cite[Lemma 4.7(b)]{FM-Poly} implies that $I_x^hI_y^{k-h}$ has linear quotients for all $k\ge1$ and all $0\le h\le k$. Multiplying an ideal with linear quotients by a monomial preserves the linear quotients property. Hence ${\bf x}_{[p]}^{k-h}{\bf y}_{[q]}^hI_y^{k-h}I_x^h$ has linear quotients for all $k\ge1$ and all $0\le h\le k$. Fix $k\ge1$ and for each $0\le h\le k$ fix a linear quotients order $>_h$ of ${\bf x}_{[p]}^{k-h}{\bf y}_{[q]}^hI_y^{k-h}I_x^h$. In view of equation (\ref{eq:G-disjoint}), we order the monomials in $\mathcal{G}(I^k)$ as follows. Let $u,v\in\mathcal{G}(I^k)$. Then $u\in\mathcal{G}({\bf x}_{[p]}^{k-h}{\bf y}_{[q]}^hI_y^{k-h}I_x^h)$ and $v\in\mathcal{G}({\bf x}_{[p]}^{k-\ell}{\bf y}_{[q]}^{\ell}I_y^{k-\ell}I_x^{\ell})$ for some $h$ and $\ell$. We put
	    $$
	    u>v\quad\Longleftrightarrow\quad h<\ell\quad\text{or}\quad h=\ell\quad\text{and}\quad u>_hv.
	    $$
	    
	    Let $u,v\in\mathcal{G}(I^k)$ with $u>v$ as described before. If $h=\ell$, since ${\bf x}_{[p]}^{k-h}{\bf y}_{[q]}^hI_y^{k-h}I_x^h$ has linear quotients, there exists $w\in\mathcal{G}({\bf x}_{[p]}^{k-h}{\bf y}_{[q]}^hI_y^{k-h}I_x^h)\subset\mathcal{G}(I^k)$ with $w>_hv$ and so also $w>v$ such that $w:v=\lcm(w,v)/v$ is a variable that divides $u:v$. Now, assume that $h<\ell$. As we computed before, we have $\sum_{i\in[p]}\deg_{x_i}(u)=pk-h>pk-\ell=\sum_{i\in[p]}\deg_{x_i}(v)$. Hence, there exists $i\in[p]$ such that $x_i$ divides $u:v$. We can write $v={\bf x}_{[p]}^{k-\ell}{\bf y}_{[q]}^\ell v_1\cdots v_{k-\ell}v_1'\cdots v_\ell'$ with $v_j\in\mathcal{G}(I_y)$ and $v_j'\in\mathcal{G}(I_x)$ for all $j$. Since $x_i$ divides $u:v$ and $I_x=({\bf x}_{[p]}/x_j:\ j\in[p])$, we must have $v_t'={\bf x}_{[p]}/x_i$ for some $t$. Up to relabeling, we may assume that $t=\ell$. Let ${\bf y}_{[q]}/(y_ry_s)\in\mathcal{G}(I_y)$ be a monomial generator. We consider the following monomial
	    \begin{align*}
	    	w\ &=\ x_i(v/(y_ry_s))\\
	    	&=\ {\bf x}_{[p]}^{k-(\ell-1)}{\bf y}_{[q]}^{\ell-1}v_1\cdots v_{k-\ell}{\bf y}_{[q]}/(y_ry_s)v_1'\cdots v_{\ell-1}'\in\mathcal{G}({\bf x}_{[p]}^{k-(\ell-1)}{\bf y}_{[q]}^{\ell-1}I_y^{k-(\ell-1)}I_x^{\ell-1}).
	    \end{align*}
	    Then $w>v$ by definition of the order $>$, and $w:v=x_i$ divides $u:v$. We conclude that $I^k$ has indeed linear quotients and (i) follows.
	\end{proof}
	
	\section{The stable set of associated primes of $I_c(G)$}\label{sec3}
	
	In this section, we prove the theorem that gives the title of the paper. We begin with the following observation.
	
	\begin{Proposition}\label{Prop:local-I_c(G)}
		Let $G$ be a finite simple graph on $[n]$, and let $F\subset[n]$. Put
		$$
		A_F\ =\ \{i\in[n]:\ \deg_{G|_F}(i)=0\ \text{and}\ \{i,j\}\in E(G)\ \text{for some}\ j\in[n]\}.
		$$
		Then
		$$
		(I_c(G))(P_F)\ =\ \begin{cases}
			I_c(G|_F)+({\bf x}_{F}/x_i\ :\ i\in A_F)&\text{if}\ (\bigcup_{e\in E(G)}e)\cap F\ne\emptyset,\\
			({\bf x}_F)&\text{otherwise}.
		\end{cases}
		$$
	\end{Proposition}
    \begin{proof}
    	Let $I=I_c(G)$, and take a monomial $u\in\mathcal{G}(I)$. Then, $u={\bf x}_{[n]}/{\bf x}_e$, for some edge $e=\{i,j\}\in E(G)$. Let $\varphi=\varphi_F$ be the map defined in the introduction. We have
    	\begin{equation}\label{eq:calculation}
    		\varphi(u)\ =\ \begin{cases}
    			\,{\bf x}_F&\textup{if}\ i,j\notin F,\\
    			\,{\bf x}_F/x_i&\textup{if}\ i\in F,\,j\notin F,\\
    			\,{\bf x}_F/x_j&\textup{if}\ i\notin F,\,j\in F,\\
    			\,{\bf x}_F/{\bf x}_e&\textup{if}\ i,j\in F.
    		\end{cases}
    	\end{equation}
    	
    	Now, we distinguish the two possible cases.\smallskip
    	
    	\textbf{Case 1.} Assume that $(\bigcup_{e\in E(G)}e)\cap F=\emptyset$. Then the calculation (\ref{eq:calculation}) shows that $\varphi(v)={\bf x}_F$ for all $v\in\mathcal{G}(I)$. Hence $I(P_F)=({\bf x}_F)$.\smallskip
    	
    	\textbf{Case 2.} Assume that $(\bigcup_{e\in E(G)}e)\cap F\ne\emptyset$. Let $H=G|_F$. Then, for any $e\in E(H)$, we have $e\in E(G)$, and equation (\ref{eq:calculation}) shows that ${\bf x}_F/{\bf x}_e\in I(P_F)$. Equation (\ref{eq:calculation}) also shows that any monomial in $I(P_F)$ has degree at least $|F|-2$. Hence, all monomials ${\bf x}_F/{\bf x}_e\in\mathcal{G}(I_c(H))$ are minimal monomial generators of $I(P_F)$. Since, by assumption, ${\bf x}_F/{\bf x}_e\in I(P_F)$ for at least one edge $e\in E(H)$, then ${\bf x}_F\notin\mathcal{G}(I(P_F))$. Hence, besides the minimal monomial generators coming from $I_c(H)$, the only possibly extra minimal monomial
    	 generators of $I(P_F)$ have degree $|F|-1$, and therefore are of the form $w={\bf x}_F/x_i$ for some $i\in F$. Equation (\ref{eq:calculation}) implies that $w$ belongs to $I(P_F)$ if and only if there exists $j\in[n]$ such that $\{i,j\}\in E(G)$. Moreover, $w$ is a minimal generator of $I(P_F)$ if and only if $\{i,j\}\notin E(H)$ for all $j\in F\setminus\{i\}$. That is, if and only if $\deg_{G|_F}(i)=0$. Altogether, $w={\bf x}_F/x_i\in\mathcal{G}(I(P_F))$ if and only if $i\in A_F$.
    \end{proof}
	
	As a consequence, we can prove Theorem \ref{ThmA}.
	\begin{proof}[Proof of Theorem \ref{ThmA}]
		First of all, note that if $i\in V(G)$ is an isolated vertex of $G$, and we let $H=G|_{[n]\setminus\{i\}}$, then $I_c(G)^k=x_i^kI_c(H)^k=(x_i^k)\cap I_c(H)^k$ for all $k\ge1$. Hence $\Ass(I_c(G)^k)=\{(x_i)\}\cup\Ass(I_c(H)^k)$ for all $k\ge1$. We may therefore assume without loss of generality that $G$ does not contain isolated vertices. Now, let $F\subset[n]$ with $|F|\ge2$. By Theorem \ref{ThmB}, $P_F\in\Ass^\infty(I_c(G))$ if and only if $\depth S_F/(I_c(G)^k(P_F))=0$ for all $k\gg0$. By Proposition \ref{Prop:local-I_c(G)}, since we assumed that $G$ does not contain isolated vertices (hence $\bigcup_{e\in E(G)}e=[n]$), we have
		$$
		I_c(G)^k(P_F)\ =\ (I_c(G)(P_F))^k\ =\ \Big(I_c(G|_F)+({\bf x}_F/x_i:\ \deg_{G|_F}(i)=0)\Big)^k.
		$$
		
		Now, we distinguish the two possible cases.\smallskip
		
		\textbf{Case 1.} Suppose that $\widetilde{b}(G|_F)=b(G|_F)$. Then $G|_F$ does not contain isolated vertices, and so $I_c(G)^k(P_F)=I_c(G|_F)^k$. By \cite[Theorem 4.1]{FM2}, $\depth S_F/I_c(G|_F)^k=b(G)=\widetilde{b}(G)$ for all $k\ge|F|-2$. Hence, $P_F\in\Ass^\infty(I_c(G))$ if and only if $\widetilde{b}(G)=0$.\smallskip
		
		\textbf{Case 2.} Suppose that $\widetilde{b}(G|_F)\ne b(G|_F)$. Then $G|_F$ contains some isolated vertex. In this case, Theorem \ref{Thm:reg-deg>=n-2}(d) and \cite[Theorem 4.1]{FM2} imply that
		\begin{align*}
			\depth S_F/(I_c(G)^k(P_F))\ &=\ \depth S_F/I_c(G|_F)^k-\mu_{|F|-1}(I_c(G)^k(P_F))\\[2pt]
			&=\ b(G)-|\{i\in F:\ \deg_{G|_F}(i)=0\}|\\
			&=\ \widetilde{b}(G).
		\end{align*}
		Hence, once again $P_F\in\Ass^\infty(I_c(G))$ if and only if $\widetilde{b}(G)=0$.\smallskip
		
		It remains to prove that $\astab(I_c(G))\le n-2$. To this end, let $P_F\in\Ass^\infty(I_c(G))$. If $F=\{i\}$ is a singleton, then $i$ is an isolated vertex of $G$, and so $P_F\in\Ass(I_c(G)^k)$ for all $k\ge1$. Otherwise, let $|F|\ge2$. Then $\depth S_F/(I_c(G)^k(P_F))=0$ for all $k\gg0$. By Theorem \ref{Thm:reg-deg>=n-2}(e) and \cite[Theorem 4.1]{FM2}, we have $\depth S_F/(I_c(G)^k(P_F))=0$ for all $k\ge|F|-2$. Hence
		$$
		\astab(I_c(G))\ \le\ \max\{1,\max_{\substack{P_F\in\Ass^\infty(I_c(G))\\ |F|\ge2}}\{|F|-2\}\}\ \le\ n-2,
		$$
		as desired.
	\end{proof}
	
	The proof of the theorem gives the following additional information.
	\begin{Corollary}\label{Cor:astabP_FI_c(G)}
		Let $G$ be a finite simple graph on $[n]$. Let $F\subset[n]$ such that $|F|\ge2$ and $\widetilde{b}(G|_F)=0$. Then
		$$
		P_F\in\Ass(I_c(G)^k),\quad\text{for all}\ k\ge\max\{1,|F|-2\}.
		$$
	\end{Corollary}
	
	\section{Consequences}\label{sec4}
	
	Given an integer $t$ and a graph $H$, we define $tH$ to be the graph consisting of $t$ disjoint copies of $H$. The following three graphs will play a crucial role.
	\begin{enumerate}
		\item[(i)] The complete graph on $n$ vertices is the graph $K_n$ with $V(K_n)=[n]$ and edge set $E(K_n)=\{\{i,j\}:\ 1\le i<j\le n\}$.\smallskip
		\item[(ii)] The path on $n$ vertices is the graph $P_n$ with $V(P_n)=[n]$ and edge set $E(P_n)=\{\{i,i+1\}:\ 1\le i<n\}=\{\{1,2\},\{2,3\},\dots,\{n-1,n\}\}$.\smallskip
		\item[(iii)] The cycle on $n$ vertices is the graph $C_n$ with $V(C_n)=[n]$ and edge set $E(C_n)=\{\{1,2\},\{2,3\},\dots,\{n-1,n\},\{n,1\}\}$.
	\end{enumerate}\medskip
    
	Recall that the $k$-th symbolic power of a monomial ideal $I\subset S$ is defined as
	$$
	I^{(k)}\ =\ \bigcap_{P_F\in\Ass(I)}I^k(P_F)S
	$$
	We always have $I^k\subset I^{(k)}$ for all $k\ge1$, but equality may fail.
	
	When $I$ is a squarefree monomial ideal, then $I^k=I^{(k)}$ for some $k\ge1$ if and only if $\Ass(I^k)=\Ass(I)$, see the proof of \cite[Theorem 1.4.6]{HHBook}. As a consequence of this simple fact and Theorem \ref{ThmA} we can give an alternative proof of \cite[Theorem 2.9]{RS}.
	
	\begin{Corollary}\label{Cor:RoySaha}
		Let $G$ be a finite simple graph on $[n]$. The following conditions are equivalent.
		\begin{enumerate}
			\item[\textup{(a)}] $I_c(G)^k=I_c(G)^{(k)}$ for all $k\ge1$.
			\item[\textup{(b)}] $I_c(G)^k=I_c(G)^{(k)}$ for some $k\ge2$.
			\item[\textup{(c)}] $G$ is isomorphic to one of the graphs $K_2$, $K_3$, $P_3$, $2K_2$, $P_4$ or $C_4$ with $($possibly$)$ some isolated vertices.
		\end{enumerate}
	\end{Corollary}
    \begin{proof}
    	(a) $\Rightarrow$ (b) is true for any ideal.\smallskip
    	
    	(b) $\Rightarrow$ (c) Suppose that $I_c(G)^k=I_c(G)^{(k)}$ for some $k\ge2$. As we explained before, this holds if and only if $\Ass(I_c(G)^k)=\Ass(I_c(G))$. It is a simple exercise in graph theory to show that if $G$ does not satisfy the statement (c), then $V(G)$ contains a subset $F$ with $|F|=4$ such that $\widetilde{b}(G|_F)=0$. Suppose for a contradiction that (c) does not hold. Then we can find such a $4$-subset $F\subset[n]$. By Corollary \ref{Cor:astabP_FI_c(G)}, $P_F\in\Ass(I_c(G)^\ell)$ for all $\ell\ge2$. Hence, $P_F\in\Ass(I_c(G)^k)$. By \cite[Theorem 2.1]{FM1} we have $P_F\in\Ass(I_c(G)^k)\setminus\Ass(I_c(G))$. A contradiction. Hence (c) follows.\smallskip
    	
    	(c) $\Rightarrow$ (a) It is straightforward to see that if $G$ satisfies the statement (c) then $\Ass(I_c(G)^k)=\Ass(I_c(G))$ for all $k\ge1$. Hence (a) follows.
    \end{proof}
	
	Our next application regards the computation of the $\v$-function of $I_c(G)$.\smallskip
	
	Following \cite{CSTVV20}, given a homogeneous ideal $I\subset S$ and a prime $P\in\Ass(I)$, the local $\v$-number of $I$ at $P$ is defined as the least degree of a homogeneous element $f\in S$ such that $I:f=P$, see also \cite{BMS24,Conca23,FPack1,F2023,FMar,FS2,FSPack,FSPackA,FSComparison}. The global $\v$-number of $I$ is then defined as
	$$
	\v(I)\ =\ \min_{P\in\Ass(I)}\v_P(I).
	$$
	
	Let $P\in\Ass^\infty(I)$. Let $\astab_P(I)$ be the least integer $k_P>0$ such that $P\in\Ass(I^k)$ for all $k\ge k_P$. Then one can consider the function $\v_P(I^k)$ for $k\ge k_P$. This is called the (local) $\v_P$-function of $I$ at $P$. The (global) $\v$-function of $I$ is simply the function $k\mapsto\v(I^k)$. It is proved independently in \cite{Conca23} and \cite{FS2}, that the local $\v$-functions and the $\v$-function of a homogeneous ideal $I$ are eventually linear functions. Moreover, in \cite[Theorem 4.1]{FS2}, it is proved that the slope of the $\v$-function of $I$ is the initial degree of $I$. The least integer $k_0>0$ such that $\v(I^k)$ is a linear function for all $k\ge k_0$, is called the \textit{index of $\v$-stability} of $I$ and is denoted by $\vstab(I)$.
	
	It is expected (\cite{FS2} and \cite[Conjecture 2.6]{F2023}) that
	\begin{enumerate}
		\item[(i)] $\v(I^k)<\reg I^k$ for any homogeneous ideal $I\subset S$ and all $k\gg0$.
		\item[(ii)] If $I\subset S$ is a monomial ideal with linear powers, then $\v(I^k)=\reg I^k-1$ for all $k\ge1$.
	\end{enumerate}

    Both conjectures (i) and (ii) have been proved for several classes of monomial ideals, including the family of edge ideals (see the results contained in the paper \cite{F2023}, and \cite[Theorem 5.1]{F2023}, \cite[Theorem 5.1]{BMS24} and \cite[Corollary 5.7]{FMar} for the proofs of the conjectures (i) and (ii) in the case of edge ideals).\smallskip
	
	Next, we establish conjectures (i) and (ii) for complementary edge ideals.
	
	\begin{Theorem}\label{Thm:v-numb-I_c(G)}
		Let $G$ be a finite simple graph on $[n]$. The following statements hold.
		\begin{enumerate}
			\item[\textup{(a)}] For all $k\ge1$,
			$$
			\v(I_c(G)^k)\ =\ \begin{cases}
				(n-2)k&\text{if}\ G=tK_2,\ \text{with}\ t\ge2,\\
				(n-2)k-1&\text{otherwise}.
			\end{cases}
			$$
			\item[\textup{(b)}] $\v(I_c(G)^k)<\reg I_c(G)^k$ for all $k\ge1$.\smallskip
			\item[\textup{(c)}] If $I_c(G)$ has linear powers, then $\v(I_c(G)^k)=\reg I_c(G)^k-1$ for all $k\ge1$.
		\end{enumerate}
	\end{Theorem}
	\begin{proof}
		(a) By \cite[Proposition 2.2]{F2023}, it follows that $\v(I_c(G)^k)\ge(n-2)k-1$ for all $k\ge1$. Therefore, when $G\ne tK_2$ to prove that equality holds for all $k\ge1$, it suffices to find a monomial $u_k$, for each $k\ge1$, such that $I_c(G)^k:u_k\in\Ass(I_c(G)^k)$ and $\deg(u_k)=(n-2)k-1$.
		
		Assume that $G$ contains an isolated vertex $i$. Then $(x_i)\in\Ass(I_c(G)^k)$ for all $k\ge1$. Let $e\in E(G)$, then $i\notin e$. We consider the monomial $u_k=({\bf x}_{[n]\setminus e})^k/x_i$ for all $k\ge1$. Note that $x_iu_k\in I_c(G)^k$ for all $k\ge1$. Now, let $v$ be a monomial with $\supp(v)\subset[n]\setminus\{i\}$. Then $\deg_{x_i}(vu_k)=k-1$. Hence $vu_k\notin I_c(G)^k$. We conclude that $I_c(G):u_k=P_{\{i\}}=(x_i)$, as desired.
		
		Suppose now that $G$ does not contain isolated vertices. Let $G^c$ be the complementary graph of $G$. That is, the graph with the same vertex set of $G$ whose edges are the non-edges. Furthermore, let $\mathcal{T}(G)$ be the set of triangles (also called $3$-cliques) of $G$. By \cite[Theorem 2.1]{FM1}, since $G$ does not have isolated vertices, we have
		$$
		\Ass(I_c(G))\ =\ \{P_e:\ e\in E(G^c)\}\cup\{P_T:\ T\in\mathcal{T}(G)\}.
		$$
		
		First, suppose that $G$ contains a triangle $T=\{p,q,r\}$. Consider the monomial $u_k=({\bf x}_{[n]}/(x_px_qx_r))({\bf x}_{[n]}/(x_px_q))^{k-1}$ with $k\ge1$. Note that $\deg(u_k)=(n-2)k-1$. We claim that $I_c(G)^k:u_k=P_T$. It is clear that $P_T\subset I_c(G)^k:u_k$. For instance,
		$$
		x_pu_k\ =\ ({\bf x}_{[n]}/(x_qx_r))({\bf x}_{[n]}/(x_px_q))^{k-1}\in I_c(G)^k.
		$$
		To prove the reverse inclusion $I_c(G)^k:u_k\subset P_T$, we show that if $v$ is a monomial with $\supp(v)\subset[n]\setminus\{p,q,r\}$, then $v\notin I_c(G)^k:u_k$. Suppose on the contrary that $v\in I_c(G)^k:u_k$. Then, there exist edges $e_1,\dots,e_k\in E(G)$ such that
		\begin{equation}\label{eq:divide-edges}
				w\ =\ \prod_{i=1}^k{\bf x}_{[n]\setminus e_i}\quad\text{divides}\quad vu_k\ =\ v({\bf x}_{[n]}/(x_px_qx_r))({\bf x}_{[n]}/(x_px_q))^{k-1}.
		\end{equation}
	    It follows that
		$$
		\deg_{x_p}(w)\ =\ k-|\{j\in[k]:\ p\in e_j\}|\ \le\ \deg_{x_p}(v)+\deg_{x_p}(u_k)\ =\ 0,
		$$
		because $x_p$ does not divide $v$ and $u_k$. Since $|\{j\in[k]:p\in e_j\}|\le k$, the above inequality implies that $|\{j\in[k]:p\in e_j\}|=k$. Hence $p\in e_j$ for all $j\in[k]$. Similarly, $q\in e_j$ for all $j\in[k]$. Hence $e_1=\dots=e_k=\{p,q\}$. It follows that $\deg_{x_r}(w)=k$. This $x_r$-degree should be less than or equal to $\deg_{x_r}(vu_k)=\deg_{x_r}(v)+\deg_{x_r}(u_k)=k-1$, because $r\notin\supp(v)$. But this is not possible. Hence $v\notin I_c(G)^k:u_k$. We conclude that $I_c(G)^k:u_k=P_T$.
		
		Now, assume that $G$ contains no triangles. Since $G$ does not have isolated vertices, either $G$ contains an induced $3$-path $P_3$ or else $G=tK_2$ with $t\ge2$ (since $n\ge3$).
		
		Suppose first that $G$ contains an induced $3$-path. Say, that $\{p,q\},\{q,r\}\in E(G)$, but $e=\{p,r\}\notin E(G)$. Hence $P_e\in\Ass(I_c(G))$. Consider again the monomial $u_k=({\bf x}_{[n]}/(x_px_qx_r))({\bf x}_{[n]}/(x_px_q))^{k-1}$ with $k\ge1$. Note that $\deg(u_k)=(n-2)k-1$. We claim that $I_c(G)^k:u_k=P_e$. As before, $P_e\subset I_c(G)^k:u_k$. Conversely, we show that if $v$ is a monomial with $\supp(v)\subset[n]\setminus\{p,r\}$, then $v\notin I_c(G)^k:u_k$. Suppose on the contrary this was the case. Then, equation (\ref{eq:divide-edges}) would be again valid. From this equation, we derive that all edges $e_1,\dots,e_k$ contain $p$. Since $\{p,r\}\notin E(G)$, it follows that $\deg_{x_r}(w)=k$. But this $x_r$-degree should be less than or equal to $\deg_{x_r}(vu_k)=k-1$, which is absurd. We conclude that $I_c(G)^k:u_k=P_e$.
		
		Now, let $G=tK_2$ with $t\ge2$. Say $E(G)=\{\{1,2\},\{3,4\},\dots,\{2t-1,2t\}\}$. First we show that there is no monomial $u\in S$ of degree $(2t-2)k-1$ such that $I_c(G)^k:u$ is a monomial prime ideal. Suppose for a contradiction that such a monomial $u$ exists. Then $I_c(G)^k:u$ should be generated by variables. Up to relabeling, we may assume that $x_1\in I_c(G)^k:u$. Then $x_1u\in I_c(G)$. Since $\deg(x_1u)=(2t-2)k$, then $x_1u\in\mathcal{G}(I_c(G)^k)$. Therefore, there exist edges $e_1,\dots,e_k\in E(G)$ such that $x_1u=\prod_{i=1}^k{\bf x}_{[2t]\setminus e_i}$. Since $x_1$ divides this monomial, we have $1\notin e_i$ for at least one index $i\in[k]$. Say $e_i=e_1=\{3,4\}$. Therefore $u={\bf x}_{[2t]\setminus\{1,3,4\}}(\prod_{i=2}^k{\bf x}_{[2t]\setminus e_i})$. Note that $x_3x_4\in I_c(G)^k:u$ because $x_3x_4u=x_2{\bf x}_{[2t]\setminus\{1,2\}}(\prod_{i=2}^{k}{\bf x}_{[2t]\setminus e_i})\in I_c(G)^k$. We claim that $x_3\notin I_c(G)^k:u$ and $x_4\notin I_c(G)^k:u$. This shows that $x_3x_4$ is a minimal generator of $I_c(G)^k:u$ and so $I_c(G)^k:u$ is not generated by variables. If $x_3\in I_c(G)^k:u$, then we should have
		$$
		x_3u\ =\ {\bf x}_{[2t]\setminus\{1,4\}}(\prod_{i=2}^{k}{\bf x}_{[2t]\setminus e_i})\ =\ {\bf x}_{[2t]}^k/((x_1x_4){\bf x}_{e_2}\cdots {\bf x}_{e_k})\in I_c(G)^k.
		$$
		Since $\deg(x_3u)=(2t-2)k$, this is equivalent to $w=(x_1x_4){\bf x}_{e_2}\cdots {\bf x}_{e_k}\in I(G)^k$. Let $w'={\bf x}_{e_2}\cdots {\bf x}_{e_k}$. Since $\mathcal{G}(I_c(G))=\{{\bf x}_{\{1,2\}},\dots,{\bf x}_{\{2t-1,2t\}}\}$, each minimal monomial generator of $v\in I(G)^k$ is such that $\deg_{x_1}(v)+\deg_{x_2}(v)$ is an even integer. On the other hand, $\deg_{x_1}(w)+\deg_{x_2}(w)=1+\deg_{x_1}(w')+\deg_{x_2}(w')$ is an odd integer, because $w\in\mathcal{G}(I_c(G)^{k-1})$. Thus $w\notin I_c(G)^k$, and so $x_3\notin I_c(G)^k:u$, as desired.
		
		Summarizing our discussion thus far, we proved that $\v(I_c(G)^k)\ge(2t-2)k$ for all $k\ge1$. Now, we prove that equality holds. To this end, consider the monomial $u_k={\bf x}_{[2t]\setminus\{1,3\}}({\bf x}_{[2t]\setminus\{1,2\}})^{k-1}$. Note that $x_1u_k=x_4{\bf x}_{[2t]\setminus\{3,4\}}({\bf x}_{[2t]\setminus\{1,2\}})^{k-1}\in I_c(G)^k$. Similarly $x_3u_k\in I_c(G)^k$. Hence $(x_1,x_3)\subset I_c(G)^k:u_k$. To prove the reverse inclusion, consider a monomial $v\in S$ with $\supp(v)\subset[2t]\setminus\{1,3\}$. Suppose for a contradiction that $vu_k\in I_c(G)^k$, then there exist edges $e_1,\dots,e_k\in E(G)$ such that $w=\prod_{i=1}^k{\bf x}_{[2t]\setminus e_i}$ divides $vu_k$. Since $\deg_{x_1}(vu_k)=0$ it follows that $\deg_{x_1}(w)=0$, and so $e_1=\dots=e_k=\{1,2\}$. Hence $\deg_{x_3}(w)=k$. However, this $x_3$-degree should be less than or equal to $\deg_{x_3}(vu_k)=k-1$, an absurd. It follows that $I_c(G)^k:u_k=(x_1,x_3)$ as desired. We conclude that $\v(I_c(G)^k)=(2t-2)k$ for all $k\ge1$, and statement (a) follows.\smallskip
		
		(b) Let $c=c(G)$. By \cite[Theorem 4.1]{FM1} or also Theorem \ref{ThmC}(c)(ii) in this paper, it follows indeed that $\v(I_c(G)^k)<\reg I_c(G)^k$ for all $k\ge1$.\smallskip
		
		(c) Assume that $I_c(G)$ has linear powers. Then $\reg I_c(G)^k=(n-2)k$ for all $k\ge1$. Moreover, \cite[Theorem B]{FM} (or also Theorem \ref{ThmC}(c)(ii)) yields $c(G)=1$. Part (a) implies that $\v(I_c(G)^k)=(n-2)k-1=\reg I_c(G)^k-1$ for all $k\ge1$.
	\end{proof}
	
	\begin{Corollary}
		We have $\vstab(I_c(G))=1$ for any graph $G$.
	\end{Corollary}
	
	It would be of great interest to determine explicitly all the local $\v$-functions of a complementary edge ideal, as well as those of an edge ideal.\medskip

	We conclude the paper with the following remark.\smallskip
	
	An ideal $I$ in a Noetherian ring is said to satisfy the \textit{strong persistence property} if $I^{k+1}:I=I^k$ for all $k\ge1$. When $I$ satisfies the strong persistence property, then it satisfies the persistence property as well. Edge ideals satisfy the strong persistence property, see \cite[Lemma 2.12]{MMV}.
	
	Besides Theorem \ref{ThmB}, other methods are available to prove that an ideal satisfies the persistence property. For instance, Herzog and Qureshi \cite[Theorem 1.4]{HQ} proved that if the Rees algebra $\mathcal{R}(I)=\bigoplus_{k\ge0}I^kt^k\subset S[t]$ of a monomial ideal $I\subset S$ satisfies Serre's condition $(S_2)$, then $I$ satisfies the (strong) persistence property.
	
	It is expected that the Rees algebra of a complementary edge ideal is always Cohen-Macaulay (\cite[Conjecture 4.7]{FM2}). If this is true, then any complementary edge ideal would satisfy the strong persistence property, and this would yield another proof of Theorem \ref{ThmA}. At present, whether $I_c(G)$ always satisfies the strong persistence property is an open question.\medskip
	
	\textbf{Acknowledgments.} A. Ficarra was supported by the Grant JDC2023-051705-I funded by MICIU/AEI/10.13039/501100011033 and by the FSE+ and also by INDAM (Istituto Nazionale Di Alta Matematica). We would like to thank Kamalesh Saha for some useful discussions regarding Theorem \ref{Thm:v-numb-I_c(G)}.
\end{document}